\documentclass{amsart}

\usepackage{tikz}
\usepackage{pgf}
\usetikzlibrary{arrows,shapes,snakes,automata,backgrounds,petri,fit}
\usepackage[english]{babel}
\usepackage{amsmath}
\usepackage{amssymb}
\usepackage{amsthm}
\usepackage[all]{xy}
\usepackage[latin1]{inputenc}        
\usepackage{amscd}
\usepackage{mathrsfs}
\usepackage[mathcal]{eucal}
\usepackage{float}
\author{G. Chiaselotti, G. Marino, C.Nardi}
\address{Chiaselotti Giampiero\\Dipartimento di Matematica, Universit\'a della Calabria, Via Pietro Bucci,
Cubo 30B, 87036 Arcavacata di Rende (CS), Italy.}
\email{chiaselotti@unical.it}
\address{Marino Giuseppe\\Dipartimento di Matematica, Universit\'a della Calabria, Via Pietro Bucci,
Cubo 30C, 87036 Arcavacata di Rende (CS), Italy.}
\email{gmarino@unical.it}
\address{Nardi Caterina\\Dipartimento di Matematica, Universit\'a della Calabria, Via Pietro Bucci,
Cubo 30C, 87036 Arcavacata di Rende (CS), Italy.}
\email{nardi@mat.unical.it}
\title[A minimum problem for finite sets]{A minimum problem for finite sets of real numbers with non-negative sum}
\date{\today}

\newtheorem{inizio}{Lemma}[section]
\newtheorem{theorem}[inizio]{Theorem}

\newtheorem{lemma}[inizio]{Lemma}

\newtheorem{definition}[inizio]{Definition}

\newtheorem{o-problem}[inizio]{Open Problem}

\newtheorem*{teo-L}{Theorem}

\newtheorem*{corollary-s}{Corollary}
\theoremstyle{definition}
\newtheorem{example}[inizio]{Example}

\begin{document}
\abstract
Let $n$ and $r$ be two integers such that $0 < r \le n$; we denote by $\gamma(n,r)$ [$\eta(n,r)$] the minimum [maximum] number of the non-negative partial sums of a sum $\sum_{1=1}^n a_i \ge 0$, where $a_1, \cdots, a_n$ are $n$ real numbers arbitrarily chosen in such a way that $r$ of them are non-negative and the remaining $n-r$ are negative. Inspired by some interesting extremal combinatorial sum problems raised by Manickam, Mikl\"os and Singhi in 1987 \cite{ManMik87} and 1988 \cite{ManSin88} we study the following two problems:

\noindent$(P1)$ {\it which are the values of $\gamma(n,r)$ and $\eta(n,r)$ for each $n$ and $r$, $0 < r \le n$?}

\noindent$(P2)$ {\it if $q$ is an integer such that $\gamma(n,r) \le q \le \eta(n,r)$, can we find $n$ real numbers $a_1, \cdots, a_n$, such that $r$ of them are non-negative and the remaining $n-r$ are negative with $\sum_{1=1}^n a_i \ge 0$, such that the number of the non-negative sums formed from these numbers is exactly $q$?}

\noindent We prove that the solution of the problem $(P1)$ is given by $\gamma(n,r) = 2^{n-1}$ and $\eta(n,r) = 2^n - 2^{n-r}$. We provide a partial result of the latter problem showing that the answer is affirmative for the weighted boolean maps. With respect to the problem $(P2)$ such maps (that we will introduce in the present paper) can be considered a generalization of the multisets $a_1, \cdots, a_n$ with $\sum_{1=1}^n a_i \ge 0$. More precisely we prove that for each $q$ such that $\gamma(n,r) \le q \le \eta(n,r)$ there exists a weighted boolean map having exactly $q$ positive boolean values.
\endabstract

\maketitle

\section{Introduction}
In \cite{ManMik87} and \cite{ManSin88} Manickam, Mikl\"os and Singhi raised several interesting extremal combinatorial sum problems, two of which will be described below. Let $n$ and $r$ be two integers such that $0 < r \le n$; we denote by $\gamma(n,r)$ [$\eta(n,r)$] the minimum [maximum] number of the non-negative partial sums of a sum $\sum_{1=1}^n a_i \ge 0$, when $a_1, \cdots, a_n$ are $n$ real numbers arbitrarily chosen in such a way that $r$ of them are non-negative and the remaining $n-r$ are negative. Put $A(n) = \min \{\gamma(n,r) : 0 < r \le n \}$. In \cite{ManMik87} the authors answered the following question:

\noindent$(Q1)$ {\it which is the value of $A(n)$?}

\noindent In Theorem 1 of \cite{ManMik87} they found that $A(n)=2^{n-1}$. On the other side, from the proof of Theorem 1 of \cite{ManMik87} also it follows that $\gamma(n,r) \ge 2^{n-1}$ for each $r$ and that $\gamma(n,1) \le 2^{n-1}$; therefore $\gamma(n,1) = 2^{n-1}$ (since $2^{n-1} = A(n) \le \gamma(n,1) \le 2^{n-1}$). It is natural to set then the following problem which is a refinement of $(Q1)$:

\noindent$(P1)$  {\it which are the values of $\gamma(n,r)$ and $\eta(n,r)$ for each $n$ and $r$ with $0 < r \le n$?}

\noindent In the first part of this paper we solve the problem $(P1)$ and we prove (see Theorem \ref{firstTheorem}) that $\gamma(n,r) = 2^{n-1}$ and $\eta(n,r) = 2^n - 2^{n-r}$ for each positive integer $r \le n$.

\noindent A further question that the authors raised in \cite{ManMik87} is the following:

\noindent$(Q2)$ {\it ``We do not know what is the range of the possible numbers of the non-negative partial sums of a non-negative $n$-element sum. The minimum is $2^{n-1}$ as it was proven and the maximum is obviously $2^n - 1$ but we do not know which are the numbers between them for which we can find reals $a_1, \cdots, a_n$ with $\sum_{1=1}^n a_i \ge 0$ such that the number of the non-negative sums formed from these numbers is equal to that number.''}

\noindent The following problem is a  natural refinement of (Q2):

\noindent$(P2)$ {\it If $q$ is an integer such that $\gamma(n,r) \le q \le \eta(n,r)$, can we find $n$ real numbers $a_1, \cdots, a_n$, such that $r$ of them are non-negative and the remaining $n-r$ are negative with $\sum_{1=1}^n a_i \ge 0$, such that the number of the non-negative sums formed from these numbers is exactly $q$?}\\
\noindent In the latter part of this paper (see Theorem \ref{mainResult}) we give a partial solution to the problem $(P2)$.

To be more precise in the formulation of the problems that we study and to better underline the links with some interesting problems raised in \cite{ManMik87} and \cite{ManSin88}, it will be convenient identify a finite set of real numbers with an appropriate real valued function. Let then $n$ and $r$ be two fixed integers such that $0 < r \le n$ and let $I_n =\{ 1,2 \cdots , n \}$ (we call $I_n$ the index set). We denote by $W(n,r)$ the set of all the functions $f:I_n \to \mathbb{R}$ such that $\sum_{x\in I_n} f(x) \ge 0$ and $|\{x \in I_n: f(x) \ge 0 \}|=r$.  If $f\in W(n,r)$ we set $\alpha (f) = |\{ Y \subseteq I_n : \sum_{y \in Y} f(y) \ge 0 \}|$. It is easy to observe that $\gamma (n,r)=\min \{ \alpha (f) : f \in W(n,r) \}$ and $\eta (n,r)=\max \{ \alpha (f) : f \in W(n,r) \}$. We can reformulate the problem $(P2)$ in an equivalent way using the functions terminology instead of the sets terminology:

\noindent$(P2)$\,\,{\it If $q$ is an integer such that $\gamma(n,r) \le q \le \eta(n,r)$, does there exist a function $f \in W(n,r)$ with the property that $\alpha(f)=q$?}\\
\noindent To solve the problem $(P1)$ and (partially) $(P2)$, we use some abstract results on a particular class of lattices introduced in \cite{BisChias} and \cite{BisChias2}. In this paper we substantially continue the research project started in \cite{Chias02}, which is the attempt to solve some extremal sum problems raised in \cite{ManMik87} and \cite{ManSin88} and further studied in \cite{bhatta-2003},\cite{bhatt-thesis}, \cite{bier-man}, \cite{BisChias2}, \cite{ChiasInfMar08}, \cite{chias-mar-2002}, \cite{manic-88}, \cite{manic-91}.

\section{A Partial Order on the Subsets of $I_n$}
 Of course if we take two functions $f,g \in W(n,r)$ such that $f(I_n)=g(I_n)$, then $\alpha(f)=\alpha(g)$. This implies that if we define on $W(n,r)$ the equivalence relation $f \sim g$ iff $f(I_n)=g(I_n)$ and we denote by $[f]$ the equivalence class of a function $f\in W(n,r)$, then the definition $\beta([f])=\alpha(f)$ is well placed. It is also clear that it holds $\gamma (n,r)=\min \{ \beta ([f]) : [f] \in W(n,r)/\sim \}$ and $\eta (n,r)=\max \{ \beta ([f]) : [f] \in W(n,r)/\sim \}$. Now, when we take an equivalence class $[f]\in W(n,r)/\sim$, there is a unique $f^*\in [f]$ such that
 \begin{equation}\label{decreasing(f*)}
 f^*(r) \ge \cdots \ge f^*(1) \ge 0 > f^*(r+1) \ge \cdots \ge f^*(n)
 \end{equation}
 We can then identify the quotient set $W(n,r)/\sim$ with the subset of all the functions $f^* \in W(n,r)$ that satisfy the condition (\ref{decreasing(f*)}). This simple remark conducts us to rename the indexes of $I_n$ as follows: $\tilde{r}$ instead of $r$, $\dots$, $\tilde{1}$ instead of $1$, $\overline{1}$ instead of $r+1$, $\dots$, $\overline{n-r}$ instead of $n$. Therefore, if we set $I(n,r)=\{\tilde{1}, \cdots, \tilde{r}, \overline{1}, \cdots, \overline{n-r}\}$, we can identify the quotient set $W(n,r)/\sim$ with the set of all the functions $f:I(n,r) \to \mathbb{R}$ that satisfy the following two conditions:

\begin{equation}\label{weightCondition(f)}
f(\tilde{r}) + \cdots + f(\tilde{1}) + f(\overline{1}) + \cdots + f(\overline{n-r}) \ge 0
\end{equation}\\
and\\
\begin{equation}\label{decreasing(f)}
f(\tilde{r}) \ge \cdots \ge f(\tilde{1}) \ge 0 > f(\overline{1}) \ge \cdots \ge f(\overline{n-r})
\end{equation}

\noindent Now, if a generic function $f:I(n,r) \to \mathbb{R}$ that satisfies (\ref{weightCondition(f)}) and (\ref{decreasing(f)}) is given,  we are interested to find all the subsets $Y \subseteq I(n,r)$ such that $\sum_{y\in Y}f(y) \ge 0$. This goal becomes then easier if we can have an appropriate partial order $\sqsubseteq$ on the power set $\mathcal{P}(I(n,r))$ "compatible" with the total order of the partial sums inducted by $f$, i.e. a partial order $\sqsubseteq$ that satisfies the following {\it monotonicity property}: if $Y,Z \in \mathcal{P}(I(n,r))$ then $\sum_{z \in Z} f(x) \le \sum_{y \in Y} f(y)$ whenever $Z \sqsubseteq Y$. To have such a partial order $\sqsubseteq$ on $\mathcal{P}(I(n,r))$ that has the monotonicity property we must introduce a new formal symbol that we denote by $0^§$. We add this new symbol to the index set $I(n,r)$, so we set $A(n,r)=I(n,r)\cup\{0^§\}$. We introduce on $A(n,r)$ the following total order:
$$\overline{n-r} \prec \cdots \prec \overline{2} \prec \overline{1} \prec 0^§ \prec \tilde {1} \prec \tilde{2}\prec \cdots \prec \tilde{r}$$

\noindent If $i,j \in A(n,r),$ then we write : $i \preceq j$ for $i=j$ or $i \prec j$.
We denote by $S(n,r)$ the set of all the formal expressions $i_1 \cdots  i_r | j_1 \cdots  j_{n-r}$ (hereafter called {\it strings}) that satisfy the following properties:

\noindent i)  $i_1, \cdots, i_r \in \{ \tilde{1}, \cdots, \tilde{r}, 0^§ \},$

\noindent ii) $j_1, \cdots, j_{n-r} \in \{ \overline{1}, \cdots, \overline{n-r}, 0^§ \},$

\noindent iii) $i_1\succeq \cdots \succeq i_r \succeq 0^§ \succeq j_1 \succeq \cdots \succeq j_{n-r},$

\noindent iv) the unique element which can be repeated is $0^§$.\\
\noindent In the sequel we often use the lowercase letters $u, w, z, ...$  to denote a generic string in $S(n,r)$. Moreover to make smoother reading, in the numerical examples the formal symbols which appear in a string will be written without  $\;\tilde{}\;$ $\;\bar{}\;$ and $\;^§\;$ ; in such way the vertical bar $|$ will indicate that the symbols on the left of $|$ are in
$\{ \tilde{1}, \cdots, \tilde{r}, 0^§ \}$ and the symbols on the right of $|$ are elements in $\{ 0^§, \overline{1}, \cdots, \overline{n-r} \}.$\\
\noindent For example, if $n=3$ and $r=2$, then $A(3,2)= \{ \tilde{2} \succ \tilde{1} \succ 0^§ \succ \overline{1} \}$ and
 $S(3,2)= \{ 21|0, \,\,\, 21|1,\,\,\, 10|0,\,\,\, 20|0,\,\,\, 10|1,\,\,\, 20|1,\,\,\, 00|1,\,\,\, 00|0 \}$.

\noindent Note that there is a natural bijective set-correspondence $*: w \in S(n,r) \mapsto w^* \in \mathcal{P}(I(n,r))$ between $S(n,r)$ and $\mathcal{P}(I(n,r))$ defined as follows: if $w = i_1 \cdots  i_r | j_1 \cdots  j_{n-r} \in S(n,r)$ then $w^*$ is the subset of $I(n,r)$ made with the elements $i_k$ and $j_l$ such that $i_k \neq 0^§$ and $j_l \neq 0^§$. For example, if $w=4310|013 \in S(7,4)$, then $w^*=\{ \tilde{1}, \tilde{3}, \tilde{4}, \overline{1}, \overline{3} \}$. In particular, if $w=0 \cdots 0|0 \cdots 0$ then $w^*=\emptyset$.

\noindent Now, if $v=i_1 \cdots  i_r |j_1 \cdots  j_{n-r}$ and $w=i'_1 \cdots  i'_r | j'_1 \cdots  j'_{n-r}$ are two strings in $S(n,r)$, we define: $v \sqsubseteq w$ iff  $i_1 \preceq i'_1, \cdots , i_r \preceq i'_r, j_1 \preceq j'_1, \cdots , j_{n-r} \preceq j'_{n-r}$.

\noindent It is easily seen that:

\noindent$1)$ $(S(n,r), \sqsubseteq)$ is a finite distributive (hence also graded) lattice with minimum element $0\cdots 0|12\cdots(n-r)$ and maximum element $r(r-1)\cdots21|0\cdots0$;

\noindent $2)$ $(S(n,r), \sqsubseteq)$ has the following unary complementary operation $c$:

\noindent $(p_1\cdots p_k0\cdots0|0\cdots0q_1\cdots q_l)^c = p'_1\cdots p'_{r-k}0\cdots0|0\cdots0q'_1\cdots q'_{n-r-l}$,

\noindent where $\{p'_1,\cdots ,p'_{r-k}\}$ is the usual complement of $\{p_1,\cdots,p_k\}$ in $\{\tilde{1}, \cdots, \tilde{r}\}$, and\\
$\{q'_1,\cdots ,q'_{n-r-l}\}$ is the usual complement of $\{q_1,\cdots,q_l\}$ in $\{\overline{1}, \cdots, \overline{n-r}\}$ (for example, in $S(7,4)$, we have that $(4310|001)^c = 2000|023$).

Since we have the formal necessity to consider functions $f$ defined on the extended set $A(n,r)$ instead of on the indexes set $I(n,r)$, then  we will put $f(0^§)=0$. Precisely we can identify the quotient set $W(n,r)/\sim$ with the set $WF(n,r)$, defined by

$WF(n,r)=\{f: A(n,r) \to \mathbb{R}:f(\tilde{r}) \ge \cdots \ge f(\tilde{1}) \ge f(0^§)=0 > f(\overline{1}) \ge \cdots \ge f(\overline{n-r})$ and $f(\tilde{r}) + \cdots + f(\tilde{1}) + f(\overline{1}) + \cdots + f(\overline{n-r}) \ge 0\}$

\noindent We call an element of $WF(n,r)$ a $(n,r)$-{\it weight function} and if $f\in WF(n,r)$ we will continue to set $\alpha (f) :=|\{ Y \subseteq I(n,r) : \sum_{y \in Y} f(y) \ge 0 \}|$. Therefore, with these last notations we have that $\gamma (n,r)=\min \{ \alpha (f) : f \in WF(n,r) \}$, $\eta (n,r)=\max \{ \alpha (f) : f \in WF(n,r) \}$ and the question $(P2)$ becomes equivalent to the following:

\noindent $(P2)$\,\,{\it If $q$ is an integer such that $\gamma(n,r) \le q \le \eta(n,r)$, does there exist a function $f \in WF(n,r)$ with the property that $\alpha(f)=q$?}

\section{Boolean Maps induct by Weight Functions}

We denote by $\bf{2}$ the boolean lattice composed of a chain with 2 elements that we denote $N$ (the minimum element) and $P$ (the maximum element). A {\it Boolean map} (briefly BM) {\it on} $S(n,r)$ is a map $A : dom(A)\subseteq S(n,r) \to {\bf 2}$, in particular if $dom(A)=S(n,r)$ we also say that $A$ is a {\it Boolean total map} (briefly BTM) {\it on} $S(n,r)$. If $A$ is BM on $S(n,r)$, we set $S_A ^+ (n,r) = \{ w \in dom(A) : A(w)=P \}$.

If $f\in WF(n,r)$, the {\it sum function} $\Sigma_f : S(n,r) \to \mathbb{R}$ induced by $f$ on $S(n,r)$ is the function that associates to $w=i_r \cdots i_1 |j_1 \cdots j_{n-r}\in S(n,r)$ the real number $\Sigma_f (w)=f(i_1)+ \cdots + f(i_r) + f(j_1) + \cdots + f(j_{n-r})$ and therefore we can associate to  $f\in WF(n,r)$ the map $A_f : S(n,r) \to \bf{2}$ setting
$$
A_f (w) = \left\{ \begin{array}{lll}
                  P & \textrm{if} & \Sigma_f (w) \ge 0 \,\,\textrm{and}\,\, w \neq 0\cdots 0|0\cdots 0\\
                  N & \textrm{if} & w =0\cdots 0|0\cdots 0\\
                  N & \textrm{if} & \Sigma_f (w) < 0
                  \end{array} \right.
$$

\noindent Let us note that $|S^+_{A_f}(n,r)|=|\{ w \in S(n,r) : A_f(w)=P \}|= \alpha(f),$ and so \\
\noindent $\gamma(n,r)=\min\{|S^+_{A_f}(n,r)|:f\in WF(n,r)\},$ $\eta(n,r)=\max\{|S^+_{A_f}(n,r)|:f\in WF(n,r)\}.$

\noindent Our goal is now to underline that some properties of such maps simplify the study of our problems. It is easy to observe that the map $A_f$ has the following properties:

\begin{itemize}
\item[(i)] $A_{f}$ is order-preserving;\\
\item[(ii)] If $w \in S(n,r)$ is such that $A_{f}(w)=N$, then  $A_{f}(w^c) =P;$ \\
\item[(iii)] $A_{f}(10 \cdots 0|0 \cdots 0)=P$, $A(0 \cdots 0|0 \cdots 0)=N$ and $A_{f}(r \cdots 21|12 \cdots (n-r))=P.$
\end{itemize}

\begin{example}

Let $f$ be the following $(5,3)$-weight function :

\begin{equation*}
   f:\begin{array}{ccccccccc}
    \tilde{3} & \tilde{2} & \tilde{1} & \overline{1} & \overline{2} & \\
    \downarrow & \downarrow & \downarrow & \downarrow & \downarrow & \\
       1         &     1      &      0.9   &     -0.8   &   -2.1   & \\
  \end{array}
\end{equation*}

\noindent We represent the map $A_f$ defined on  $S(5,3)$ by using the Hasse diagram of this lattice, on which we color green the nodes where the map $A_f$ assumes value P and red the nodes where it assumes value N:

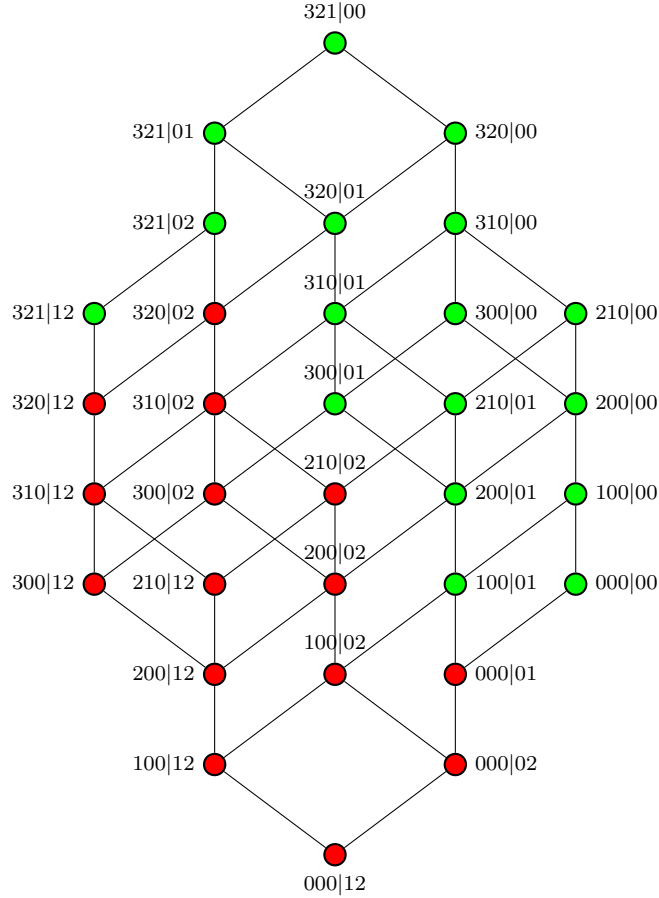
\begin{center}

\begin{tikzpicture}
 [inner sep=1.0mm,
 placeg/.style={circle,draw=black!100,fill=green!100,thick},
 placer/.style={circle,draw=black!100,fill=red!100,thick},scale=0.4]

 \path
{ (0,0)node(1a) [placer,label=270:{\footnotesize$000|12$}]{}}

{ (-4,3)node(1b) [placer,label=180:{\footnotesize$100|12$}]{}}
{(4,3)node(2b) [placer,label=0:{\footnotesize$000|02$}]{}}

{ (-4,6)node(1c) [placer,label=180:{\footnotesize$200|12$}]{}}
{ (0,6)node(2c) [placer,label=90:{\footnotesize$100|02$}]{}}
{  (4,6)node(3c) [placer,label=0:{\footnotesize$000|01$}]{}}

{ (-8,9)node(1d) [placer,label=180:{\footnotesize$300|12$}]{}}
{ (-4,9)node(2d) [placer,label=180:{\footnotesize$210|12$}]{}}
{ (0,9)node(3d) [placer,label=90:{\footnotesize$200|02$}]{}}
{ (4,9)node(4d) [placeg,label=0:{\footnotesize$100|01$}]{}}
{ (8,9)node(5d) [placeg,label=0:{\footnotesize$000|00$}]{}}

{ (-8,12)node(1e) [placer,label=180:{\footnotesize$310|12$}]{}}
{ (-4,12)node(2e) [placer,label=180:{\footnotesize$300|02$}]{}}
{ (0,12)node(3e) [placer,label=90:{\footnotesize$210|02$}]{}}
{ (4,12)node(4e) [placeg,label=0:{\footnotesize$200|01$}]{}}
{ (8,12)node(5e) [placeg,label=0:{\footnotesize$100|00$}]{}}

{ (-8,15)node(1f) [placer,label=180:{\footnotesize$320|12$}]{}}
{ (-4,15)node(2f) [placer,label=180:{\footnotesize$310|02$}]{}}
{ (0,15)node(3f) [placeg,label=90:{\footnotesize$300|01$}]{}}
{ (4,15)node(4f) [placeg,label=0:{\footnotesize$210|01$}]{}}
{ (8,15)node(5f) [placeg,label=0:{\footnotesize$200|00$}]{}}

{ (-8,18)node(1g) [placeg,label=180:{\footnotesize$321|12$}]{}}
{ (-4,18)node(2g) [placer,label=180:{\footnotesize$320|02$}]{}}
{ (0,18)node(3g) [placeg,label=90:{\footnotesize$310|01$}]{}}
{ (4,18)node(4g) [placeg,label=0:{\footnotesize$300|00$}]{}}
{ (8,18)node(5g) [placeg,label=0:{\footnotesize$210|00$}]{}}

{  (-4,21)node(1h) [placeg,label=180:{\footnotesize$321|02$}]{}}
{ (0,21)node(2h) [placeg,label=90:{\footnotesize$320|01$}]{}}
{(4,21)node(3h) [placeg,label=0:{\footnotesize$310|00$}]{}}

{ (-4,24)node(1i) [placeg,label=180:{\footnotesize$321|01$}]{}}
{ (4,24)node(2i) [placeg,label=0:{\footnotesize$320|00$}]{}}

{ (0,27)node(1j) [placeg,label=90:{\footnotesize$321|00$}]{}};

 \draw (1a)--(1b); 
 \draw (1a)--(2b);

 \draw (1b)--(1c); 
 \draw (1b)--(2c);
 \draw (2b)--(2c);
 \draw (2b)--(3c);

 \draw (1c)--(1d); 
 \draw (1c)--(2d);
 \draw (1c)--(3d);
 \draw (2c)--(3d);
 \draw (2c)--(4d);
 \draw (3c)--(4d);
 \draw (3c)--(5d);

 \draw (1d)--(1e); 
 \draw (1d)--(2e);
 \draw (2d)--(1e);
 \draw (2d)--(3e);
 \draw (3d)--(2e);
 \draw (3d)--(3e);
 \draw (3d)--(4e);
 \draw (4d)--(4e);
 \draw (4d)--(5e);
 \draw (5d)--(5e);

 \draw (1e)--(1f); 
 \draw (1e)--(2f);
 \draw (2e)--(2f);
 \draw (2e)--(3f);
 \draw (3e)--(2f);
 \draw (3e)--(4f);
 \draw (4e)--(3f);
 \draw (4e)--(4f);
 \draw (4e)--(5f);
 \draw (5e)--(5f);

 \draw (1f)--(1g); 
 \draw (1f)--(2g);
 \draw (2f)--(2g);
 \draw (2f)--(3g);
 \draw (3f)--(3g);
 \draw (3f)--(4g);
 \draw (4f)--(3g);
 \draw (4f)--(5g);
 \draw (5f)--(4g);
 \draw (5f)--(5g);

 \draw (1g)--(1h); 
 \draw (2g)--(1h);
 \draw (2g)--(2h);
 \draw (3g)--(2h);
 \draw (3g)--(3h);
 \draw (4g)--(3h);
 \draw (5g)--(3h);

 \draw (1h)--(1i); 
 \draw (2h)--(1i);
 \draw (2h)--(2i);
 \draw (3h)--(2i);

 \draw (1i)--(1j); 
 \draw (2i)--(1j);

 \end{tikzpicture}

\end{center}

\end{example}

 Note that if we have a generic Boolean total map $A : S(n,r) \to {\bf 2}$ which has the properties (i), (ii) and (iii) of $A_f$, i.e. the following:

\begin{itemize}
\item[(BM1)] $A$ is order-preserving;\\
\item[(BM2)] If $w \in S(n,r)$ is such that $A(w)=N$, then  $A(w^c) =P;$ \\
\item[(BM3)] $A(10 \cdots 0|0 \cdots 0)=P$, $A(0 \cdots 0|0 \cdots 0)=N$ and $A(r \cdots 21|12 \cdots (n-r))=P;$
\end{itemize}

\noindent in general there does not exist a function $f\in WF(n,r)$ such that $A_f=A$ (see \cite{BisChias2} for a 
counterexample).

We denote by $\mathcal{W_+}(S(n,r),{\bf 2})$ the set of all the maps $A : S(n,r) \to {\bf 2}$ which satisfy (BM1) and (BM2) and by $\mathcal{W}_+(n,r)$ the subset of all the maps in $\mathcal{W_+}(S(n,r),{\bf 2})$ which satisfy also (BM3).  We also set $\gamma^{*} (n,r):=\min\{|S_{A}^{+}(n,r)|:A\in \mathcal{W}_{+}(n,r)\}$ and $\eta^{*} (n,r):=\max\{|S_{A}^{+}(n,r)|:A\in \mathcal{W}_{+}(n,r)\}$. Let us observe that $\gamma^{*}(n,r) \le \gamma(n,r) \le \eta(n,r) \le \eta^{*}(n,r)$. A natural question raises at this point :

\noindent $(Q)$:\,\,{\it If $q$ is an integer such that $\gamma^{*} (n,r) \le q \le \eta^{*}(n,r)$, does there exist a map $A \in \mathcal{W}_+(n,r)$ with the property that $|S_{A}^{+}(n,r)|=q$?}

\noindent The question $(Q)$ is the analogue of $(P2)$ expressed in terms of Boolean total maps on $S(n,r)$ instead of $(n,r)$-weight functions, and if we are able to respond to $(Q)$ we provide also a partial answer to $(P2)$. In section \ref{IntermediateValuesSection} we give an affirmative answer to the question $(Q)$ and also we give a constructive method to build the map $A$.

\section{Main results}\label{IntermediateValuesSection}

In the sequel of this paper we adopt the  classical terminology and notations usually used in the context of the partially ordered sets  (see \cite{dav-pri-2002} and \cite{stanley-vol1} for the general aspects on this subject). If $Z \subseteq S(n,r)$, we will set $\downarrow Z = \{ x \in S(n,r) : \exists \,\,\, z \in Z \,\,\, \textrm{such that} \,\,\, x \sqsubseteq z \} $, $\uparrow Z = \{ x \in S(n,r) : \exists \,\,\, z \in Z \,\,\, \textrm{such that} \,\,\, z \sqsubseteq x \}$. In particular, if $z \in S(n,r)$, we will set $\downarrow z =\downarrow \{z\} =\{ x \in S(n,r) : z \sqsupseteq x \}$, $\uparrow z =\uparrow \{z\} =\{ x \in S(n,r) : z \sqsubseteq x \}$. $Z$ is called a {\it down-set} of $S(n,r)$ if for $z \in Z$ and $x \in S(n,r)$ with $z \sqsupseteq x$, then $x \in Z$. $Z$ is called an {\it up-set} of $S(n,r)$ if for $z \in Z$ and $x\in S(n,r)$ with $z \sqsubseteq x$, then $x \in Z$. $\downarrow Z$ is the smallest down-set of $S(n,r)$ which contains $Z$ and $Z$ is a down-set of $S(n,r)$ if and only if $Z=\downarrow Z$. Similarly, $\uparrow Z$ is the smallest up-set of $S(n,r)$ which contains $Z$ and $Z$ is an up-set in $S(n,r)$ if and only if $Z=\uparrow Z$.

\hspace{1cm}

\begin{theorem}\label{firstTheorem}
If $n$ and $r$ are two integers such that $0<r<n$, then: $\gamma(n,r)= \gamma^{*}(n,r)= 2^{n-1}$ and $\eta(n,r)= \eta^{*}(n,r)= 2^{n}-2^{n-r}$.
\end{theorem}
\begin{proof} Assume that $0<r<n$. We denote by $S_1 (n,r)$ the sublattice of $S(n,r)$ of all the strings $w$ of the form
$ w=i_1 \cdots i_r | j_1 \cdots j_{n-r-1} (n-r)$, with $j_1 \cdots j_{n-r-1} \in \{ 0^§, \overline{1}, \cdots, \overline{n-r-1} \}$ and  by $S_2 (n,r)$ the sublattice of $S(n,r)$ of all the strings $w$ of the form
$ w=i_1 \cdots i_r |0 j_2 \cdots j_{n-r}$, with $j_2 \cdots j_{n-r} \in \{ 0^§, \overline{1}, \cdots, \overline{n-r-1} \}$. It is clear that $S(n,r)=S_1(n,r)\dot\bigcup S_2(n,r)$ and  $S_1(n,r) \cong S_2(n,r) \cong S(n-1,r)$.
\noindent We consider now the following further sublattices of $S(n,r)$:

\noindent  $S_1^+(n,r):=\left\{ w\in S_{1}(n,r):w=r(r-1)\ldots 21| j_{1}\ldots j_{n-r-1}(n-r)\right\}$

\noindent $S_1^{\pm}(n,r):=\left\{ w\in S_{1}(n,r):w=i_{1}\ldots i_{r-1}0| j_{1}\ldots j_{n-r-1}(n-r),i_{1}\succ0^§\right\} $

 \noindent $S_1^-(n,r):=\left\{ w\in S_{1}(n,r):w=0\ldots0| j_{1}\ldots j_{n-r-1}(n-r)\right\} $

\noindent $S_2^+(n,r):=\left\{ w\in S_{2}(n,r):w=r(r-1)\ldots 21| 0 j_{2}\ldots j_{n-r}\right\}$

\noindent $S_2^{\pm}(n,r):=\left\{ w\in S_{2}(n,r):w=i_{1}\ldots i_{r-1}0|0j_{2}\ldots j_{n-r},i_{1}\succ 0^§\right\}$

 \noindent $S_2^-(n,r):=\left\{ w\in S_{2}(n,r):w=0\ldots0| 0j_{2}\ldots j_{n-r}\right\} $

\noindent It occurs immediately that: $S_i(n,r)=S_i^+(n,r){\dot{\bigcup}}S_i^{\pm}(n,r)\dot{\bigcup}S_i^-(n,r)$ , for $i=1,2$ and $S_i^{\pm}(n,r)$ is a distributive sublattice of $S_i(n,r)$ with $2^{n-1}-2\cdot2^{n-r-1}=2^{n-1}-2^{n-r}$ elements, for $i=1,2$.

\noindent Now we consider the following $(n,r)$-weight function $f:A(n,r)\rightarrow\mathbb{R}$

 \begin{equation*}
 f:\begin{array}{cccccccc}
\widetilde{r} & \ldots & \widetilde{1} & 0^§ & \overline{1} & \ldots & \overline{(n-r-1)} & \overline{(n-r)}\\
\downarrow & \downarrow & \downarrow & \downarrow & \downarrow & \downarrow & \downarrow & \downarrow\\
(n-r) & \ldots & (n-r) & 0 & -1 & \ldots & -1 & (n-r)(1-r)-1
\end{array}
\end{equation*}

\noindent  Then it follows that $\Sigma_{f}:S(n,r)\rightarrow\mathbb{R}$ is such that

 \begin{equation*}
 \Sigma_{f}(w)\begin{cases}
\begin{array}{c}
\geq0\\
<0\end{array} & \begin{array}{c}
{\rm if}\,\, w\in S_2^{\pm}(n,r)\\
{\rm if}\,\, w\in S_1^{\pm}(n,r)\end{array}\end{cases}
 \end{equation*}

\noindent It means that the boolean map $A_{f}\in \mathcal{W}_{+}(n,r)$ is such that:

\begin{equation*}
A_{f}(w)=\begin{cases}
\begin{array}{c}
P\\
N\end{array} & \begin{array}{c}
{\rm if}\,\, w\in S_2^{\pm}(n,r)\\
{\rm if}\,\, w\in S_1^{\pm}(n,r)\end{array}\end{cases}
\end{equation*}

\noindent This shows that: $\left|S_{A_{f}}^{+}(n,r)\right|=\left|S_1^+(n,r)\dot{\bigcup}S_2^+(n,r)\dot{\bigcup}S_2^{\pm}(n,r)\right|= 2^{n-r-1}+2^{n-r-1}+2^{n-1}-2\cdot2^{n-r-1}=2^{n-1}$.

\noindent In \cite{ManMik87} (Theorem 1) it has been proved that $\gamma(n,1)=2^{n-1}$  and  $\gamma(n,r)\geq2^{n-1}$. Since $\gamma(n,r) \ge \gamma^{*}(n,r)$, using a technique similar to that used in the proof of the theorem 1 of \cite{ManMik87}, it easily follows that $\gamma^{*}(n,r)\geq2^{n-1}$. As shown above, it results that $|S_{A_f}^{+}(n,r)|=2^{n-1}$. Hence $2^{n-1}\ge\gamma(n,r)\geq\gamma^{*}(n,r)\geq2^{n-1}$, i.e. $\gamma(n,r)=\gamma^{*}(n,r)=2^{n-1}$. This prove the first part of theorem, it remains to prove the latter part.

We consider now the following $(n,r)-$positive weight function $g:A(n,r)\rightarrow\mathbb{R}$:

 \begin{equation*}
g: \begin{array}{cccccccc}
\widetilde{r} & \ldots & \widetilde{1} & 0^§ & \overline{1} & \ldots & \overline{(n-r-1)} & \overline{(n-r)}\\
\downarrow & \downarrow & \downarrow & \downarrow & \downarrow & \downarrow & \downarrow & \downarrow\\
1 & \ldots & 1 & 0 & \frac{-1}{n-r} & \ldots & \frac{-1}{n-r} & \frac{-1}{n-r}
\end{array}
\end{equation*}

\noindent It results then that the sum $\Sigma_{g}:S(n,r)\rightarrow\mathbb{R}$ is such that $\Sigma_{g}(w)\ge 0 \,\,{\rm if}\,\, w \in S_1^{\pm}(n,r)\dot{\bigcup} S_2^{\pm}(n,r)$, i.e. the boolean map $A_{g}\in \mathcal{W}_{+}(n,r)$ is such that $A_{g}(w)=P \,\,{\rm if}\,\, w\in S_1^{\pm}(n,r)\dot{\bigcup} S_2^{\pm}(n,r)$. This shows that $\left|S_{A_{g}}^{+}(n,r)\right|=\\
\left|S_1^{\pm}(n,r)\dot{\bigcup}S_2^{\pm}(n,r)\dot{\bigcup}S_1^+(n,r)\dot{\bigcup}S_2^+(n,r)\right|=
(2^{n-1}-2^{n-r})+(2^{n-1}-2^{n-r})+2^{n-r-1}+2^{n-r-1}=2^{n}-2^{n-r}$.

\noindent On other hand, it is clear that for any $ A\in \mathcal{W}_+(n,r)$ it results $A(w)=P$ for each $w\in S_1^+(n,r)\dot{\bigcup}S_2^+(n,r)$ and $A(w)=N$ for each $w\in S_1^-(n,r)\dot{\bigcup}S_2^-(n,r)$. The number $(2^{n}-2^{n-r})$ is the biggest number of values $P$ that a boolean map $A\in \mathcal{W}_{+}(n,r)$ can assume. Hence, since $\eta(n,r) \le \eta^{*}(n,r)$, we have $2^{n}-2^{n-r} = \left|S_{A_g}^{+}(n,r)\right| \le \eta(n,r) \le \eta^{*}(n,r) \le 2^{n}-2^{n-r}$, i.e. $\eta(n,r) = \eta^{*}(n,r) = 2^{n}-2^{n-r}$. This conclude the proof of the Theorem \ref{firstTheorem}.
\end{proof}

\noindent To better visualize the previous result, we give a numerical example on a specific Hasse diagram. Let $n=6$ and $r=2$ and let $f$ as given in previous theorem, i.e.

\begin{equation*}
 f:\begin{array}{cccccccc}
\widetilde{2} & \widetilde{1} & 0^§           & \overline{1} & \overline{2} & \overline{3}  & \overline{4}\\
\downarrow    & \downarrow    &    \downarrow & \downarrow   & \downarrow     & \downarrow    & \downarrow  \\
    4         &     4         & 0             &      -1            & -1             & -1        & -5
\end{array}
\end{equation*}

\noindent In the figure below we have shown the Hasse diagram of the lattice S(6,2), where $S_1^+(n,r)$ is black; $S_1^{\pm}(n,r)$ is violet; $S_1^-(n,r)$ is red; $S_2^+(n,r)$ is blue; $S_2^{\pm}(n,r)$ is brown; $S_2^-(n,r)$ is green. Therefore $A_{f}$ assume the following values:
\begin{itemize}
  \item the blue, black and brown nodes correspond to values P of $A_{f}$
  \item the violet, red and green nodes correspond to values N of $A_{f}$.
\end{itemize}
\newpage
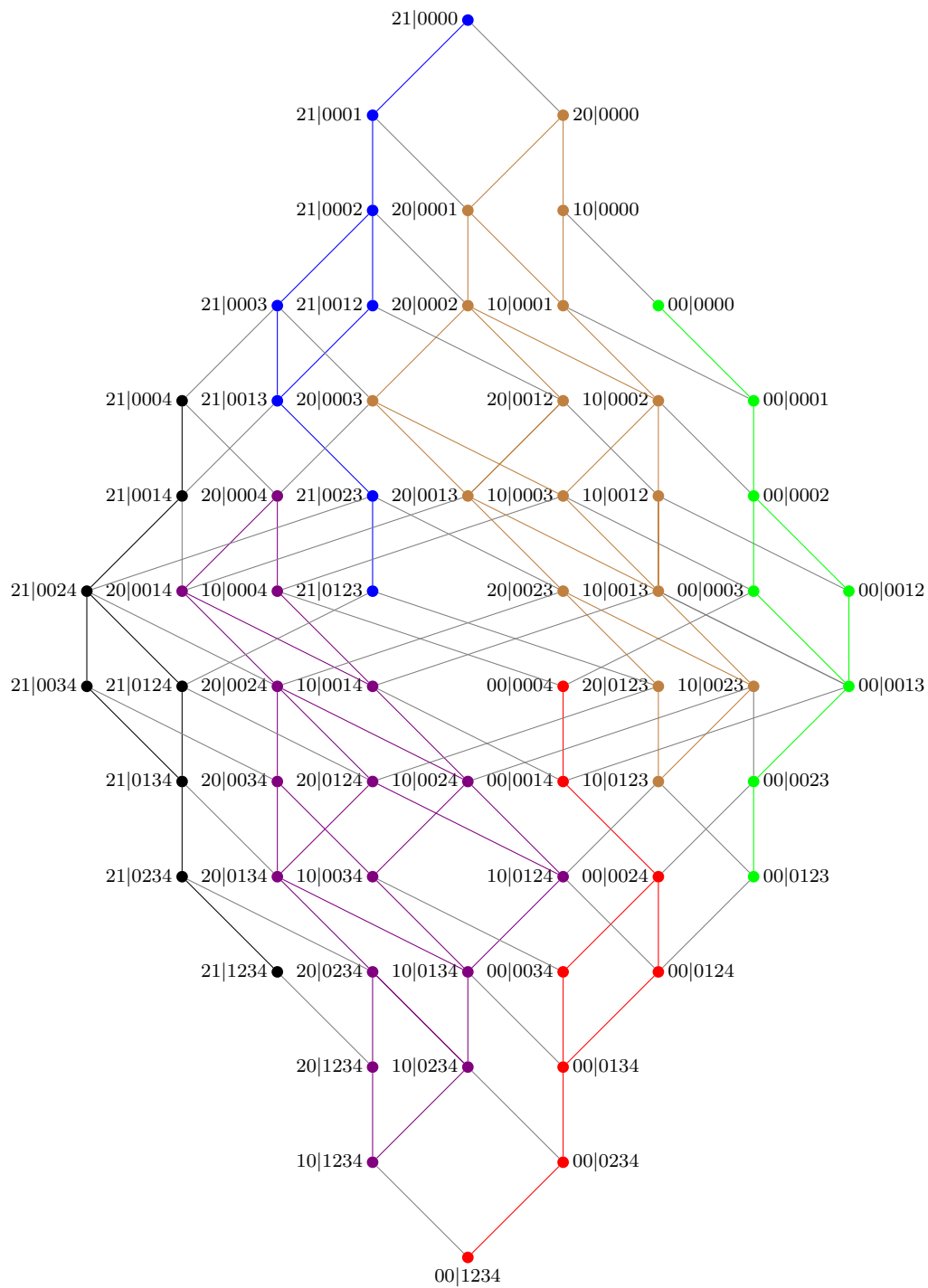
\begin{tikzpicture}
 [inner sep=0.5mm,
 place/.style={circle,draw=black!100,fill=black!100,thick},
placer/.style={circle,draw=red!100,fill=red!100,thick},
placeb/.style={circle,draw=blue!100,fill=blue!100,thick},
placeg/.style={circle,draw=green!100,fill=green!100,thick},
placep/.style={circle,draw=violet!100,fill=violet!100,thick},
placem/.style={circle,draw=brown!100,fill=brown!100,thick},scale=0.35]scale=0.35]
 \path
 (0,0)node(a1) [placer,label=270:{\footnotesize$00|1234$}]{}

 (4,4)node(b1) [placer,label=0:{\footnotesize$00|0234$}]{}
 (-4,4)node(b2) [placep,label=180:{\footnotesize$10|1234$}]{}
 (-4,8)node(c1) [placep,label=180:{\footnotesize$20|1234$}]{}
 (0,8)node(c2) [placep,label=180:{\footnotesize$10|0234$}]{}
 (4,8)node(c3) [placer,label=0:{\footnotesize$00|0134$}]{}
 (-8,12)node(d1) [place,label=180:{\footnotesize$21|1234$}]{}
 (-4,12)node(d2) [placep,label=180:{\footnotesize$20|0234$}]{}
 (0,12)node(d3) [placep,label=180:{\footnotesize$10|0134$}]{}
 (4,12)node(d4) [placer,label=180:{\footnotesize$00|0034$}]{}
 (8,12)node(d5) [placer,label=0:{\footnotesize$00|0124$}]{}
 (-12,16)node(e1) [place,label=180:{\footnotesize$21|0234$}]{}
 (-8,16)node(e2) [placep,label=180:{\footnotesize$20|0134$}]{}
 (-4,16)node(e3) [placep,label=180:{\footnotesize$10|0034$}]{}
 (4,16)node(e4) [placep,label=180:{\footnotesize$10|0124$}]{}
 (8,16)node(e5) [placer,label=180:{\footnotesize$00|0024$}]{}
 (12,16)node(e6) [placeg,label=0:{\footnotesize$00|0123$}]{}
 (-12,20)node(f1) [place,label=180:{\footnotesize$21|0134$}]{}
 (-8,20)node(f2) [placep,label=180:{\footnotesize$20|0034$}]{}
 (-4,20)node(f3) [placep,label=180:{\footnotesize$20|0124$}]{}
 (0,20)node(f4) [placep,label=180:{\footnotesize$10|0024$}]{}
 (4,20)node(f5) [placer,label=180:{\footnotesize$00|0014$}]{}
 (8,20)node(f6) [placem,label=180:{\footnotesize$10|0123$}]{}
 (12,20)node(f7) [placeg,label=0:{\footnotesize$00|0023$}]{}
 (-16,24)node(g1) [place,label=180:{\footnotesize$21|0034$}]{}
(-12,24) node(g2) [place,label=180:{\footnotesize$21|0124$}]{}
(-8,24)node(g3) [placep,label=180:{\footnotesize$20|0024$}]{}
(-4,24)node(g4) [placep,label=180:{\footnotesize$10|0014$}]{}
(4,24)node(g5) [placer,label=180:{\footnotesize$00|0004$}]{}
(8,24)node(g6) [placem,label=180:{\footnotesize$20|0123$}]{}
(12,24)node(g7) [placem,label=180:{\footnotesize$10|0023$}]{}
(16,24)node(g8) [placeg,label=0:{\footnotesize$00|0013$}]{}
(-16,28)node(h1) [place,label=180:{\footnotesize$21|0024$}]{}
(-12,28) node(h2) [placep,label=180:{\footnotesize$20|0014$}]{}
(-8,28)node(h3) [placep,label=180:{\footnotesize$10|0004$}]{}
(-4,28)node(h4) [placeb,label=180:{\footnotesize$21|0123$}]{}
(4,28)node(h5) [placem,label=180:{\footnotesize$20|0023$}]{}
(8,28)node(h6) [placem,label=180:{\footnotesize$10|0013$}]{}
(12,28)node(h7) [placeg,label=180:{\footnotesize$00|0003$}]{}
(16,28)node(h8) [placeg,label=0:{\footnotesize$00|0012$}]{}
(-12,32)node(i1) [place,label=180:{\footnotesize$21|0014$}]{}
 (-8,32)node(i2) [placep,label=180:{\footnotesize$20|0004$}]{}
 (-4,32)node(i3) [placeb,label=180:{\footnotesize$21|0023$}]{}
 (0,32)node(i4) [placem,label=180:{\footnotesize$20|0013$}]{}
 (4,32)node(i5) [placem,label=180:{\footnotesize$10|0003$}]{}
 (8,32)node(i6) [placem,label=180:{\footnotesize$10|0012$}]{}
 (12,32)node(i7) [placeg,label=0:{\footnotesize$00|0002$}]{}
 (-12,36)node(l1) [place,label=180:{\footnotesize$21|0004$}]{}
 (-8,36)node(l2) [placeb,label=180:{\footnotesize$21|0013$}]{}
 (-4,36)node(l3) [placem,label=180:{\footnotesize$20|0003$}]{}
 (4,36)node(l4) [placem,label=180:{\footnotesize$20|0012$}]{}
 (8,36)node(l5) [placem,label=180:{\footnotesize$10|0002$}]{}
 (12,36)node(l6) [placeg,label=0:{\footnotesize$00|0001$}]{}
(-8,40)node(m1) [placeb,label=180:{\footnotesize$21|0003$}]{}
 (-4,40)node(m2) [placeb,label=180:{\footnotesize$21|0012$}]{}
 (0,40)node(m3) [placem,label=180:{\footnotesize$20|0002$}]{}
 (4,40)node(m4) [placem,label=180:{\footnotesize$10|0001$}]{}
 (8,40)node(m5) [placeg,label=0:{\footnotesize$00|0000$}]{}
 (-4,44)node(n1) [placeb,label=180:{\footnotesize$21|0002$}]{}
 (0,44)node(n2) [placem,label=180:{\footnotesize$20|0001$}]{}
 (4,44)node(n3) [placem,label=0:{\footnotesize$10|0000$}]{}
 (4,48)node(p2) [placem,label=0:{\footnotesize$20|0000$}]{}
 (-4,48)node(p1) [placeb,label=180:{\footnotesize$21|0001$}]{}
 (0,52)node(q1) [placeb,label=180:{\footnotesize$21|0000$}]{}
;
\draw[red] (a1)--(b1)--(c3)--(d5)--(e5)--(f5)--(g5);
\draw[red] (c3)--(d4)--(e5);

\draw[violet] (b2)--(c1)--(d2)--(e2)--(f2)--(g3)--(h2)--(i2);
\draw[violet] (b2)--(c2)--(d3)--(e3)--(f4)--(g4)--(h3)--(i2);
\draw[violet] (c2)--(d2);
\draw[violet] (d3)--(e2)--(f3)--(g3);
\draw[violet] (d3)--(e4)--(f4);
\draw[violet] (e3)--(f2);
\draw[violet] (e4)--(f3);
\draw[violet] (c2)--(d2);
\draw[violet] (f4)--(g3);
\draw[violet] (g4)--(h2);

\draw[blue] (q1)--(p1)--(n1)--(m1)--(l2)--(i3)--(h4);
\draw[blue] (n1)--(m2)--(l2);

\draw[brown] (p2)--(n3)--(m4)--(l5)--(h6)--(i6);
\draw[brown] (p2)--(n2)--(m3)--(l4)--(i4)--(h5)--(g6)--(f6);
\draw[brown] (n2)--(m4);

\draw[brown] (m3)--(l3)--(i5)--(h6);
\draw[brown] (l4)--(i4)--(h6);
\draw[brown] (m3)--(l5)--(i5);
\draw[brown] (l3)--(i4);
\draw[brown] (h5)--(g7)--(f6);
\draw[brown] (g7)--(h6);

\draw (d1)--(e1)--(f1)--(g1)--(h1)--(i1)--(l1);
\draw (f1)--(g2)--(h1);

\draw[green] (e6)--(f7)--(g8)--(h8)--(i7)--(l6)--(m5);
\draw[green](g8)--(h7)--(i7);

\draw[gray] (c1)--(d1);
\draw[gray] (d2)--(e1) ;
\draw[gray] (e5)--(f7);
\draw[gray] (d5)--(e6)--(f6) ;
\draw[gray] (e4)--(f6);
\draw[gray] (e2)--(f1) ;
\draw[gray] (f3)--(g2);
\draw[gray] (f2)--(g1) ;
\draw[gray] (f3)--(g6);
\draw[gray] (f4)--(g7) ;
\draw[gray] (f5)--(g8)--(h6);
\draw[gray] (f7)--(g7) ;
\draw[gray] (g3)--(h1) ;
\draw[gray] (g2)--(h4);
\draw[gray] (g3)--(h5) ;
\draw[gray] (g4)--(h6);
\draw[gray] (g8)--(h6) ;
\draw[gray] (h1)--(i3);
\draw[gray] (h2)--(i1) ;
\draw[gray] (h2)--(i4);
\draw[gray] (h3)--(i5) ;
\draw[gray] (h7)--(i5);
\draw[gray] (h8)--(i6) ;
\draw[gray] (i2)--(l1)--(m1);
\draw[gray] (i1)--(l2) ;
\draw[gray] (i2)--(l3);
\draw[gray] (i7)--(l5) ;
\draw[gray] (l6)--(m4);
\draw[gray] (m5)--(n3) ;
\draw[gray] (a1)--(b2);
\draw[gray] (b1)--(c2);
\draw[gray] (c3)--(d3);
\draw[gray] (d4)--(e3);
\draw[gray] (d5)--(e4);
\draw[gray] (f5)--(g4);
\draw[gray] (g5)--(h3);
\draw[gray] (q1)--(p2);
\draw[gray] (p1)--(n2);
\draw[gray] (n1)--(m3);
\draw[gray] (m2)--(l4);
\draw[gray] (m1)--(l3);
\draw[gray] (g6)--(h4);
\draw[gray] (l4)--(i6);
\draw[gray] (i3)--(h5);
\draw[gray] (g5)--(h7);

 \end{tikzpicture}

\newpage


 First to give the proof of the Theorem \ref{mainResult} we need to introduce some useful results and the concept of basis in $S(n,r)$.
\noindent In the following first lemma we show some properties of the sublattices of $S(n,r)$.
\begin{lemma}\label{propertiesSublattices}
Here hold the following properties, where $\theta= 00 \cdots 0|0 \cdots 0$ and $\Theta= r\cdots21|12\cdots(n-r)$
\begin{itemize}

\item[i)] $\uparrow \Theta=S_{1}^{+}(n,r)\dot{\bigcup} S_{2}^{+}(n,r)$\\
\item[ii)] $\downarrow \theta=S_{1}^{-}(n,r)\dot{\bigcup} S_{2}^{-}(n,r)$\\
\item[iii)] $\uparrow S_{2}^{\pm}(n,r)\subseteq S_{2}^{\pm}(n,r)\dot{\bigcup} S_{2}^{+}(n,r)$\\
\item[iv)] $\downarrow S_{1}^{\pm}(n,r)\subseteq S_{1}^{\pm}(n,r)\dot{\bigcup} S_{1}^{-}(n,r) $\\
\item[v)] $((S_{1}^{\pm}(n,r))^{c}=S_{2}^{\pm}(n,r)$
\end{itemize}
\end{lemma}

\begin{proof}
i) If $w\in (\uparrow \Theta)$ then $\Theta\sqsubseteq w$, i.e. it has the form $w=r\ldots1|j_{1}\ldots j_{n-r}$, where $j_1 \cdots j_{n-r} \in \{ 0^§, \overline{1}, \cdots, \overline{n-r} \}$; therefore $w\in S_{1}^{+}(n,r)\dot{\bigcup} S_{2}^{+}(n,r) $. If $w\in S_{1}^{+}(n,r)$, it has the form $w=r\ldots1|j_{1}\ldots j_{n-r-1}(n-r)$, where $j_1 \cdots j_{n-r-1} \in \{ 0^§, \overline{1}, \cdots, \overline{n-r-1} \}$; if $w\in S_{2}^{+}(n,r)$, it has the form $w=r\ldots1|0j_{2}\ldots j_{n-r}$, where $j_2 \cdots j_{n-r} \in \{ 0^§, \overline{1}, \cdots, \overline{n-r-1} \}$. In both cases it results that $\Theta\sqsubseteq w$, i.e. $w\in (\uparrow \Theta) $. ii) It is analogue to i).


\noindent iii) The minimum of the sublattice $S_2^{\pm}(n,r)$ is  $\alpha=10\ldots0|01\ldots(n-r-1)$; since $\uparrow S_{2}^{\pm}(n,r)\subseteq \uparrow \alpha$, it is sufficient to show that $\uparrow\alpha=S_{2}^{\pm}(n,r)\bigcup S_{2}^{+}(n,r)$. The inclusion $S_{2}^{\pm}(n,r)\bigcup S_{2}^{+}(n,r) \subseteq\uparrow\alpha$ follows by the definition of $S_{2}^{\pm}(n,r)$ and of $S_{2}^{+}(n,r)$. On the other side, if $w\in \uparrow\alpha$, it follows that $\alpha\sqsubseteq w$, i.e. $w=i_{1}\ldots i_{r}|0j_{2}\ldots j_{n-r}$, with $i_1 \succ 0^§$ and $j_2 \cdots j_{n-r} \in \{ 0^§, \overline{1}, \cdots, \overline{n-r-1} \}$. Therefore $w\in S_{2}^{\pm}(n,r)\bigcup S_{2}^{+}(n,r) $, and this proves the other inclusion.

\noindent iv) Let us consider the maximum of the sublattice $S_1^{\pm}(n,r)$, that is $t_{1}=r(r-1)\ldots20|0\ldots0(n-r)$. Since $\downarrow S_{1}^{\pm}(n,r)\subseteq\downarrow\beta$, it is sufficient to show that $\downarrow t_{1}=S_{1}^{\pm}(n,r)\bigcup S_{1}^{-}(n,r)$; this proof is similar to iii).

\noindent v) If $w\in S_{1}^{\pm}(n,r)$, it has the form $w=i_{1}\ldots i_{r-1}0|j_{1}\ldots j_{n-r-1}(n-r)$,  with $i_{1} \succ 0^§$ and $j_1 \cdots j_{n-r-1} \in \{ 0^§, \overline{1}, \cdots, \overline{n-r-1} \}$, therefore its complement has the form $w^{c}=i_{1}'\ldots i_{r-1}'0|0j_{2}'\ldots j_{n-r}'$, with $i_{1}' \succ 0^§$ and $j_2' \cdots j_{n-r}' \in \{ 0^§, \overline{1}, \cdots, \overline{n-r-1} \}$; hence $w^{c}\in S_{2}^{\pm}(n,r)$. This shows that $(S_{1}^{\pm}(n,r))^{c} \subseteq S_{2}^{\pm}(n,r)$. Now, if $w\in S_{2}^{\pm}(n,r)$, we write $w$ in the form $(w^c)^c$; then $w^c\in (S_{2}^{\pm}(n,r))^{c} \subseteq S_{1}^{\pm}(n,r)$, therefore $w\in (S_{1}^{\pm}(n,r))^{c}$. This shows that $ S_{2}^{\pm}(n,r) \subseteq (S_{1}^{\pm}(n,r))^{c} $.
\end{proof}

At this point let us recall the definition of basis for $S(n,r)$, as given in \cite{BisChias2} in a more general context. In the same way as an anti-chain uniquely determines a Boolean order-preserving map, a basis uniquely determines a Boolean map that has the properties (BM1) and (BM2) (see \cite{BisChias2} for details). Hence the concept of basis will be fundamental in the sequel of this proof.
\begin{definition}\label{def base S(n,r)}
A basis for $S(n,r)$ is an ordered couple $\langle Y_+ | Y_{-} \rangle$, where $Y_{+}$ and $Y_{-}$ are two disjoint anti-chains of $S(n,r)$ such that :
\begin{itemize}
\item[B1)] $(\downarrow Y_{+}) \bigcap (Y_{-}^c) = \emptyset;$ \\
\item[B2)] $((\uparrow Y_{+}) \bigcup \uparrow (Y_{-}^c)) \bigcap \downarrow Y_{-} = \emptyset;$ \\
\item[B3)] $S(n,r)=((\uparrow Y_{+}) \bigcup \uparrow (Y_{-}^c)) \bigcup \downarrow Y_{-}.$\\
\end{itemize}
\end{definition}

\noindent In the proof of Theorem \ref{mainResult} we will construct explicitly a such basis.

We will also use the following result that was proved in \cite{BisChias2}.
\begin{lemma}\label{basisKer(A)}
Let $\langle W_+ | W_{-} \rangle$ be a basis for $S(n,r)$. Then the map
$$A(x)=\left\{ \begin{array}{lll}
                 P & \textrm{if} & x \in \uparrow (W_+) \bigcup \uparrow (W_{-}^c) \\
                 N & \textrm{if} & x \in \downarrow W_{-}
                 \end{array} \right.$$
is such that $A \in \mathcal{W}_+(S(n,r), {\bf 2})$.
\end{lemma}

\begin{theorem}\label{mainResult}
If $q$ is a fixed integer with $2^{n-1}\leq q \leq 2^{n}-2^{n-r}$ then there exists a boolean map $A_q\in \mathcal{W}_{+}(n,r)$ such that $\left|S_{A_q}^{+}(n,r)\right|=q$.
\end{theorem}

\begin{proof}
\noindent Let $q$ be a fixed integer  such that $2^{n-1}\leq q\leq2^{n}-2^{n-r}$. We determine a specific Boolean total map $A\in \mathcal{W}_{+}(n,r)$ such that $\left|S_{A}^{+}(n,r)\right|= q$. We proceed as follows.

\noindent The case $r=1$ is proved in the previous Theorem \ref{firstTheorem}. Let us assume then $r>1$.


\noindent Since $S_1^{\pm}(n,r)$ is a finite distributive sublattice of graded lattice $S(n,r)$, also $S_1^{\pm}(n,r)$ is a graded lattice. We denote by $R$ the rank of $S_1^{\pm}(n,r)$ and with $\rho_1$ its rank function. Note that the bottom of $S_1^{\pm}(n,r)$ is $b_1=10\ldots0|1\ldots(n-r-1)(n-r)$ and the top is $t_1=r(r-1)\ldots20|0\ldots0(n-r)$.

\noindent We write $q$ in the form $ q=2^{n-1}+p$ with $0\leq p\leq2^{n}-2^{n-r}-2^{n-1}=2^{n-1}-2^{n-r}=\left|S_1^{\pm}(n,r)\right|$. We will build a map $A\in W_{+}(n,r)$ such that $\left|S_A^{+}(n,r)\right|=2^{n-1}+p$.

\noindent If $p=0$ we take $A=A_{f}$, with $f$ as in Theorem \ref{firstTheorem} and hence we have $\left|S_A^{+}(n,r)\right|=2^{n-1}$.

\noindent If $p=2^{n-1}-2^{n-r}$ we take $A=A_{g}$, with $g$ as in Theorem \ref{firstTheorem}, and  we have $\left|S_A^{+}(n,r)\right|=2^n -2^{n-r}$.

\noindent Therefore we can assume that $0<p<2^{n-1}-2^{n-r}$.

\noindent If $0\leq i\leq R$, we denote by $\Re_{i}$ the set of elements $w\in S_1^{\pm}(n,r)$ such that $\rho_1(w)=R-i$  and we also set $\beta_{i}:=|\Re_{i}|$ to simplify the notation. We write each $\Re_{i}$ in the following form : $\Re_{i} = \{v_{i1},\dots,v_{i\beta_i} \}$. If $0\leq l\leq R$ we set $\mathfrak{U}_l:=\dot{\underset{i=0,\dots,l}{\bigcup}}\Re_i$.

\noindent If $0\leq l\leq R-2$ we set $\mathfrak{B}_l:=\dot{\underset{i=l+2,\dots,R}{\bigcup}}\Re_i$ and $\mathfrak{B}_{R-1}:= \mathfrak{B}_R := \emptyset$. We can then write $p$ in the form $p=\sum_{i=0}^{k}|\Re_{i}| + s = |\mathfrak{U}_k|+s$, for some $0\leq s<|\Re_{k+1}|$ and some $0\leq k\leq R-1$. Depending on the previous number $s$ we partition $\Re_{k+1}$ into the following two disjoint subsets : $\Re_{k+1}=\{v_{(k+1)1},\ldots,v_{(k+1)s}\}\dot{\bigcup}\{v_{(k+1)(s+1)},\ldots,v_{(k+1)\beta_{k+1}}\}$, where the first subset is considered empty if $s=0$.

\noindent In the sequel, to avoid an overload of notations, we write simply $v_i$ instead of $v_{(k+1)i}$, for $i=1,\dots,\beta_{k+1}$.\\
\noindent Let us note that $S_1^{\pm}(n,r) = \mathfrak{U}_k\dot{\bigcup} \Re_{k+1} \dot{\bigcup} \mathfrak{B}_k$.

We define now the map $A:S(n,r)\rightarrow\mathbf{2}$

\begin{equation}\label{formaA}
A(w)=\begin{cases}
\begin{array}{c}
P\\
P\\
N\\
N\end{array} & \begin{array}{c}
{\rm if}\,\, w\in S_2^{\pm}(n,r)\dot{\bigcup} S_1^+(n,r) \dot{\bigcup} S_2^+(n,r)\\
{\rm if}\,\, w\in\mathfrak{U}_k\dot{\bigcup}\left\{ v_1,\ldots,v_s\right\} \\
{\rm if}\,\, w\in \mathfrak{B}_k\dot{\bigcup}\left\{ v_{s+1},\ldots,v_{\beta_{k+1}}\right\}\\
{\rm if}\,\, w\in S_1^-(n,r)\dot{\bigcup} S_2^-(n,r) \end{array}\end{cases}
\end{equation}

\noindent Let us observe that $|S_A^+(n,r)|= |S_2^{\pm}(n,r)\dot{\bigcup} S_1^+(n,r) \dot{\bigcup} S_2^+(n,r)|+|\mathfrak{U}_k\dot{\bigcup}\left\{ v_1,\ldots,v_s\right\}|=(2^{n-1} - 2^{n-r}) + 2^{n-r-1} + 2^{n-r-1} + |\mathfrak{U}_k|+s = 2^{n-1}+p = q$.

\noindent Therefore, if we show that $A\in \mathcal{W}_{+}(n,r)$, the theorem is proved. 

\noindent  We write $\Re_{k}$ in the following way: $\Re_{k}=\left\{ w_{1},\ldots,w_{t}\right\}\dot{\bigcup}\left\{w_{t+1},\ldots,w_{\beta_{k}}\right\}$, where $\left\{ w_{1},\ldots,w_{t}\right\} =\Re_{k}\bigcap\uparrow\left\{ v_{1},\ldots,v_{s}\right\} $ and $\left\{w_{t+1},\ldots,w_{\beta_{k}}\right\}=\Re_{k} \setminus \left\{ w_{1},\ldots,w_{t}\right\}$. Analogously $\Re_{k+2}=\left\{z_{1},\ldots,z_{q}\right\}\dot{\bigcup}\left\{z_{q+1},\ldots,z_{\beta_{k+2}}\right\} $, where $\left\{ z_{q+1},\ldots,z_{\beta_{k+2}}\right\} =\Re_{k+2}\bigcap\downarrow\left\{ v_{s+1},\ldots,v_{\beta_{k+1}}\right\} $ and $\left\{z_1,\ldots,z_q\right\}=\Re_{k+2} \setminus \left\{z_{q+1},\ldots,z_{\beta_{k+2}}\right\}$.\\

\noindent We can see a picture of this partition of the sublattice $S_{1}^{\pm}(n,r)$ in the following figure:

 
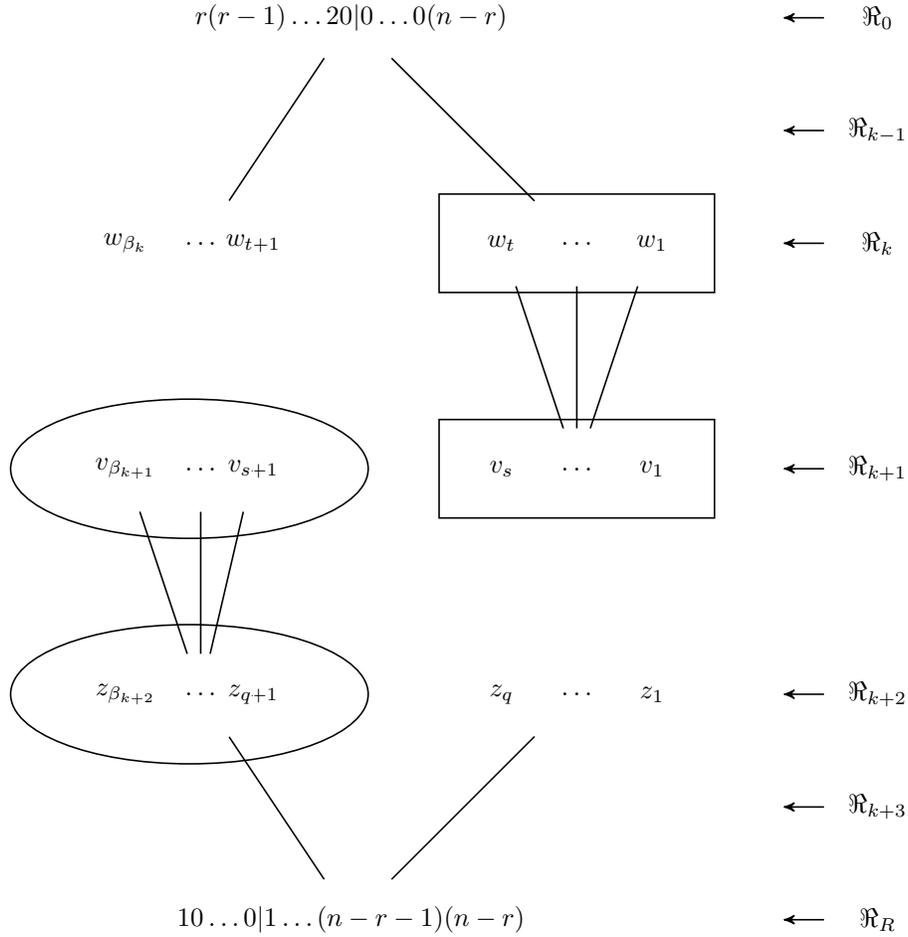
\begin{tikzpicture}[>=stealth',shorten >=1pt,auto,node distance=4cm,
                    semithick]

  \tikzstyle{initial}=[rectangle,draw=white, minimum height=3em,minimum width=4em]

\node[initial] (a) at (0,0) {$10\ldots0|1\ldots(n-r-1)(n-r)$};
\node[initial] (a1) at (7,0) {$\Re_{R}$};
\node[initial]  (a2) at (5,0) {};
\node[initial]  (b) at (2,3) {$z_{q}$};
\node[initial]  (c) at (3,3) {$\ldots$};
\node[initial]  (d) at (4,3) {$z_{1}$};
\node[initial]  (c1) at (7,3) {$\Re_{k+2}$};
\node[initial]  (c2) at (5,3) {};
\node[initial]  (b1) at (7,1.5) {$\Re_{k+3}$};
\node[initial]  (b2) at (5,1.5) {};
\node[initial]  (e) at (-1.3,3) {$z_{q+1}$};
\node[initial]  (f) at (-2,3) {$\ldots$};
\node[initial]  (g) at (-3,3) {$z_{\tiny {\beta_{k+2}}}$};
\node[draw,ellipse, fit=(e) (f) (g)] {};
\node[initial]  (h) at (-1.3,6) {$v_{s+1}$};
\node[initial]  (i) at (-2,6) {$\ldots$};
\node[initial]  (l) at (-3,6) {$v_{\tiny {\beta_{k+1}}}$};
\node[draw,ellipse, fit=(h) (i) (l)] {};
\node[initial]  (m) at (2,6) {$v_{s}$};
\node[initial]  (n) at (3,6) {$\ldots$};
\node[initial]  (o) at (4,6) {$v_{1}$};
\node[initial]  (d1) at (7,6) {$\Re_{k+1}$};
\node[initial]  (d2) at (5,6) {};
\node[draw,rectangle, fit=(m) (n) (o)] {};
\node[initial]  (p) at (2,9) {$w_{t}$};
\node[initial]  (q) at (3,9) {$\ldots$};
\node[initial]  (r) at (4,9) {$w_{1}$};
\node[initial]  (e1) at (7,9) {$\Re_{k}$};
\node[initial]  (e2) at (5,9) {};
\node[draw,rectangle, fit=(p) (q) (r)] {};
\node[initial]  (s) at (-1.3,9) {$w_{t+1}$};
\node[initial]  (t) at (-2,9) {$\ldots$};
\node[initial]  (u) at (-3,9) {$w_{\tiny {\beta_{k}}}$};
\node[initial]  (f1) at (7,10.5) {$\Re_{k-1}$};
\node[initial]  (f2) at (5,10.5) {};
\node[initial] (v) at (0,12)    {$r(r-1)\ldots 20|0\ldots0(n-r)$};
\node[initial]  (g1) at (7,12) {$\Re_{0}$};
\node[initial]  (g2) at (5,12) {};
\draw (v)--(q);
\draw (v)--(t);
\draw (n)--(r);
\draw (n)--(q);
\draw (n)--(p);
\draw (f)--(i);
\draw (f)--(h);
\draw (f)--(l);
\draw (a)--(c);
\draw (a)--(f);
\draw[->] (a1)--(a2);
\draw[->] (b1)--(b2);
\draw[->] (c1)--(c2);
\draw[->] (d1)--(d2);
\draw[->] (e1)--(e2);
\draw[->] (f1)--(f2);
\draw[->] (g1)--(g2);
\end{tikzpicture}

 Depending on $s$ and $k$, we build now a particular basis for $S(n,r)$. 

\noindent To such aim, let us consider the minimum $\alpha=10\ldots0|01\ldots(n-r-1)$ of $S_2^{\pm}(n,r)$ and the subsets $T_+:=\{v_1,\ldots,v_s,w_{t+1},\ldots,w_{\beta_{k}}\}$ and $T_{-}:=\left\{ v_{s+1},\ldots,v_{\beta_{k+1}},z_{1},\ldots,z_{q}\right\}$. Let us distinguish two cases:

 \begin{itemize}
\item[(a1)] $\alpha\in\uparrow T_+  $\\
\item[(a2)] $\alpha\notin\uparrow T_+$.

\end{itemize}

\noindent We define two different couples of subsets  as follows:\\
in the case $(a1)$ we set $Y_{+}:=T_+ \,\,{\rm and}\,\, Y_-:= T_- \dot{\bigcup} \{\theta\} $; in the case $(a2)$ we set $Y_{+}:=T_+ \bigcup \{\alpha\} \,\,{\rm and}\,\, Y_-:= T_- \dot{\bigcup} \{\theta\}.$

 \hspace{1cm}

 {\it Step 1 $\langle Y_+ | Y_{-} \rangle$ is a couple of two disjoint anti-chains of $S(n,r)$}

In both cases (a1) and (a2) it is obvious that $Y_{+}\bigcap Y_{-}=\emptyset$.\\
Case (a1): the elements $\left\{ v_{1},\ldots,v_{s}\right\}$ are not comparable between them because they have all rank $R-(k+1)$ and analogously for the elements $\left\{ w_{t+1},\ldots,w_{\beta_{k}}\right\}$ that have all rank $R-k$. Let now $v\in \left\{ v_{1},\ldots,v_{s}\right\}$ and $w\in\left\{ w_{t+1},\ldots,w_{\beta_{k}}\right\} $; then $w\notin\downarrow v$ because $\rho_1(v)<\rho_1(w)$ and $w\notin\uparrow v$ because $\left\{ w_{t+1},\ldots,w_{\beta_{k}}\right\} \bigcap\uparrow\left\{ v_{1},\ldots,v_{s}\right\} =\emptyset$ by construction. For the elements in $Y_{-}$ different from $\theta$, we can proceed as for $Y_{+}$. On the other side, we can observe that $\theta$ is not comparable with none of the elements $ v_{s+1},\ldots,v_{\beta_{k+1}}, z_{1},\ldots,z_{q}$ since these are all in $S_{1}^{\pm}(n,r)$ while $\theta \in S_2^-(n,r)$. Thus $Y_{+}$ is an anti-chain.\\
Case (a2): In this case we only must show that $\alpha$ is not comparable with none of the elements $v_{1},\ldots,v_{s}, w_{t+1},\ldots,w_{\beta_{k}}$. At first from the fact that $\alpha\notin\uparrow T_+$ it follows  $v_{i}\notin\downarrow\alpha$ for each $i=1,\ldots,s$ and $w_{j}\notin\downarrow\alpha$ for each $j=t+1,\ldots,\beta_{k}$. Moreover, the elements $v_{1},\ldots,v_{s}$ and $w_{t+1},\ldots,w_{\beta_{k}}$ are all in $S_{1}^{\pm}(n,r)$, hence they have the form $i_{1}\ldots i_{r-1}0|j_{1}\ldots j_{n-r-1}(n-r)$, with $i_{1}\succ 0^§$, while $\alpha=10\ldots0|01\ldots(n-r-1)$; so $\alpha\notin\downarrow v_ {i}$ and $\alpha\notin\downarrow w_{j}$ for each $i=1,\ldots,s$ and each $j=t+1,\ldots,\beta_{k}$.
 
 \hspace{1cm}
 
{\it Step 2 $\langle Y_+ | Y_{-} \rangle$ is a basis for $S(n,r)$}

We must see that B1), B2) and B3) hold in the both cases (a1) and (a2).\\
Case(a1):

\noindent B1) Let us begin to observe that $T_-^c = \left\{ v_{s+1}^{c},\ldots,v_{\beta_{k+1}}^{c}, z_{1}^{c},\ldots,z_{q}^{c}\right\} \subseteq\left(S_{1}^{\pm}\left(n,r\right)\right)^{c}=S_{2}^{\pm}(n,r)$. Since $\theta^{c}=\Theta$, we have then $Y_{-}^{c}=T_-^c \dot{\bigcup}\{\theta^c\}\subseteq S_{2}^{\pm}(n,r)\dot{\bigcup}\left\{\Theta\right\} \subseteq S_{2}^{\pm}(n,r)\dot{\bigcup}S_{1}^{+}(n,r)$.  On the other hand, since $\downarrow Y_+ \subseteq \downarrow S_{1}^{\pm}(n,r)$, by lemma \ref{propertiesSublattices} iv) we have also that $\downarrow Y_{+}\subseteq S_{1}^{\pm}(n,r) \dot{\bigcup}S_{1}^{-}(n,r)$. Hence $(\downarrow Y_{+}) \bigcap (Y_{-}^c) = \emptyset$. This proves B1).

\noindent B2) We show at first that $\uparrow Y_{-}^{c}\bigcap \downarrow Y_{-}=\emptyset$. Since $Y_{-}^{c}\subseteq S_{2}^{\pm}(n,r)\dot{\bigcup} \left\{ \Theta\right\}$, we have that $\uparrow Y_{-}^{c}\subseteq \uparrow(S_{2}^{\pm}(n,r))\bigcup\uparrow \Theta$. By lemma \ref{propertiesSublattices} i) and iii) we have then $\uparrow Y_{-}^{c}\subseteq S_{2}^{\pm}(n,r)\dot{\bigcup} S_{1}^{+}(n,r)\dot{\bigcup} S_{2}^{+}(n,r)$. On the other side, since $\downarrow T_- \subseteq\downarrow S_{1}^{\pm}(n,r))$, by lemma \ref{propertiesSublattices} iv) it follows that $\downarrow T_- \subseteq S_{1}^{\pm}(n,r)\dot{\bigcup} S_{1}^{-}(n,r)$. By proposition \ref{propertiesSublattices} ii), we have $\downarrow \theta = S_{1}^{-}(n,r)\dot{\bigcup} S_{2}^{-}(n,r)$. Hence it holds $\downarrow Y_- \subseteq S_{1}^{\pm}(n,r)\dot{\bigcup }S_{1}^{-}(n,r)\dot{ \bigcup} S_{2}^{-}(n,r)$. This proves that $\uparrow Y_{-}^{c}\bigcap\downarrow Y_{-}=\emptyset$. We show now that also $\uparrow Y_{+}\bigcap\downarrow Y_{-} = \emptyset$; we proceed by contradiction. Let us suppose that there exists an element  $z\in\uparrow Y_{+}\bigcap\downarrow Y_{-}$, then there are two elements $ w_{+}\in Y_{+}$ and $ w_{-}\in Y_{-}$ such that $w_{+}\sqsubseteq z\sqsubseteq w_{-}$,hence $w_{+}\sqsubseteq w_{-}$. We will distinguish the following five cases, and in each of them we will find a contradiction.

\noindent  a) $w_{+}\in\left\{ v_{1},\ldots,v_{s}\right\}$ and $w_{-}\in\left\{ v_{s+1},\ldots,v_{\beta_{k+1}}\right\} $. \\
  In this case $w_+$ and $w_-$ are two distinct elements having both rank $R-(k+1)$ and such that $w_{+}\sqsubseteq w_{-}$; it is not possible.
  
\noindent  b) $w_{+}\in\left\{ v_{1},\ldots,v_{s}\right\}$ and $w_{-}\in\left\{ z_{1},\ldots,z_{q}\right\} $.\\
  In this case $w_+$ has rank $R-(k+1)$ while $w_{-}$ has rank $R-(k+2) < R-(k+1)$, and this contradicts the condition $w_{+}\sqsubseteq w_{-}$.
  
\noindent c) $w_{+}\in\left\{ w_{t+1},\ldots,w_{\beta_{k}}\right\} $ and $w_{-}\in\left\{ v_{s+1},\ldots,v_{\beta_{k+1}}\right\}. $\\
  This case is similar to the previous because $w_+$ has rank $R-k$ while $w_-$ has rank $R-(k+1)$.
  
\noindent d) $w_{+}\in\left\{ w_{t+1},\ldots,w_{\beta_{k}}\right\} $ and $w_{-}\in\left\{ z_{1},\ldots,z_{q}\right\}.$\\
  Similar to the previous because $w_+$ has rank $R-k$ while $w_-$ has rank $R-(k+2)$.
  
 \noindent e) $w_{+}\in\left\{v_{1},\ldots,v_{s}, w_{t+1},\ldots,w_{\beta_{k}}\right\} $ and $w_{-}= \theta $.\\
  In this case the condition $w_{+}\sqsubseteq w_{-}$ implies that $w_{+}\in \downarrow \theta$; since $\downarrow \theta = S_{1}^{-}(n,r)\dot{\bigcup} S_{2}^{-}(n,r)$ this is not possible since $w_{+}\in S_{1}^{\pm}(n,r)$.

\noindent B3) Since $\alpha$ is the minimum of $S_{2}^{\pm}(n,r)$ we have that  $S_{2}^{\pm}(n,r)\subseteq\uparrow\alpha$, moreover $\alpha\in\uparrow Y_{+}$; therefore $S_{2}^{\pm}(n,r)\subseteq \uparrow\alpha\subseteq\uparrow Y_{+}$. Since $\Theta = \theta^c \in Y_-^c$, it follows that $\uparrow \Theta \subseteq \uparrow Y_{-}^{c}$. By lemma \ref{propertiesSublattices} i) we have then that $(S_{1}^{+}(n,r)\dot{\bigcup }S_{2}^{+}(n,r))=\uparrow\Theta\subseteq\uparrow Y_{-}^{c}$. By lemma \ref{propertiesSublattices} ii) we also have that $(S_{1}^{-}(n,r)\dot{\bigcup } S_{2}^{-}(n,r))=\downarrow\theta\subseteq\downarrow Y_{-}$.\\
To complete the proof of B3) let us observe that $S_{1}^{\pm}(n,r)=(\overset{k}{\underset{i=0}{\bigcup}}\Re_{i})\dot{\bigcup}\Re_{k+1}\dot{\bigcup}
(\overset{R}{\underset{j=k+2}{\bigcup}}\Re_{j})$, where $(\overset{k}{\underset{i=0}{\bigcup}}\Re_{i})\subseteq \uparrow \Re_k\subseteq \uparrow Y_{+}$ , $(\overset{R}{\underset{j=k+2}{\bigcup}}\Re_{j})\subseteq \downarrow \Re_{k+2}\subseteq \downarrow Y_{-}$.

 \noindent Moreover, since $\Re_{k+1}=\left\{ v_{1},\ldots,v_{s}\right\} \dot{\bigcup}\left\{ v_{s+1},\ldots,v_{\beta_{k+1}}\right\}$, with $\left\{ v_{1},\ldots,v_{s}\right\}\subseteq Y_+ \subseteq \uparrow Y_{+} $ and $\left\{ v_{s+1},\ldots,v_{\beta_{k+1}}\right\}\subseteq Y_- \subseteq \downarrow Y_{-}$, then  $\Re_{k+1}\subseteq(\uparrow Y_{+}\bigcup\downarrow Y_{-})$. This shows that $S_{1}^{\pm}(n,r)\subseteq(\uparrow Y_{+}\bigcup\downarrow Y_{-})$. Since the six sublattices  $S_i^+(n,r)$, $S_i^{\pm}(n,r)$ and $S_i^-(n,r)$ for $i=1,2$ are a partition of $S(n,r)$, the property B3) is proved.

Case(a2):

\noindent B1) We begin to show that $\alpha \notin Y_{-}^{c}$. In fact, it holds that $\alpha=(r(r-1)\dots20|0\dots0(n-r))^{c}$; suppose that $\alpha \in Y_{-}^{c}$ and show that we obtain a contradiction. By $\alpha \in Y_{-}^{c}$ it follows $\alpha^{c}=t_{1}\in Y_{-}$. But $t_{1}$ is the top of $S_{1}^{\pm}(n,r)$, so $Y_{-}=t_{1}$, since $Y_{-}$ is anti-chain. This means that $\Re_{k+1}=\Re_0$ and this is not possible, of course. Since $Y_+=T_+ \dot{\bigcup}\{\alpha\}$, we have $(\downarrow Y_+)\bigcap Y_-^c = ((\downarrow T_+) \bigcap Y_-^c) \bigcup ((\downarrow \alpha) \bigcap Y_-^c)$ and, moreover, as in the proof of B1) in the case (a1), we also have that $(\downarrow T_+) \bigcap Y_-^c = \emptyset$; therefore, to prove B1) it is sufficient to show that $(\downarrow \alpha) \bigcap Y_-^c = \emptyset$. As in the proof of B1) in the case (a1), we have $Y_{-}^{c} \subseteq S_{2}^{\pm}(n,r)\dot{\bigcup}\left\{\Theta\right\}$. Since $\alpha$ is the minimum of $S_2^{\pm}(n,r)$ and $\Theta \notin \downarrow \alpha$, it follows that the unique element of $\downarrow \alpha$ that can belong to $Y_-^c$ is $\alpha$, but we have before shown that this is not true.

\noindent B2) As in the previous case we have $(\uparrow Y_{-}^{c})\bigcap\downarrow Y_{-}=\emptyset$, moreover $(\uparrow Y_+) \bigcap (\downarrow Y_-) = ((\uparrow T_+)\bigcup(\uparrow \alpha))\bigcap (\downarrow Y_-) ((\uparrow T_+)\bigcap(\downarrow Y_-))\bigcup((\uparrow \alpha)\bigcap(\downarrow Y_-)).$
As in the case (a1) we have $((\uparrow T_+)\bigcap(\downarrow Y_-))=\emptyset$, therefore, to prove B2) also in the case (a2), it is sufficient to show that $((\uparrow \alpha)\bigcap(\downarrow Y_-))=\emptyset$. As in the proof of B2) in the case (a1) it results that $\downarrow Y_- \subseteq S_{2}^{\pm}(n,r)\dot{\bigcup}S_{1}^{-}(n,r)\dot{ \bigcup} S_{2}^{-}(n,r)$ and, by definition of $\alpha$, it is easy to observe that $(\uparrow \alpha) \bigcap (S_{1}^{\pm}(n,r)\dot{\bigcup} S_{1}^{-}(n,r) \dot{\bigcup} S_{2}^{-}(n,r))= \emptyset$. Hence $((\uparrow \alpha)\bigcap(\downarrow Y_-))=\emptyset$.

\noindent B3) Identical to that of case (a1).

{\it Step 3 The map $A$ defined in (\ref{formaA}) is such that $A\in \mathcal{W}_{+}(n,r)$ }

\noindent Since we have proved that $\langle Y_{+}|Y_{-}\rangle$ is a basis for $S(n,r)$, by lemma \ref{basisKer(A)} it follows that the map $A \in \mathcal W_{+}(S(n,r),{\bf 2})$ if the two following identities hold:

\begin{equation}\label{firstEqualityA}
(\uparrow Y_{+})\bigcup(\uparrow Y_{-}^c)=S_2^{\pm}(n,r)\dot{\bigcup} S_1^+(n,r) \dot{\bigcup} S_2^+(n,r)\dot{\bigcup} \mathfrak{U}_k\dot{\bigcup}\left\{ v_1,\ldots,v_s\right\}
\end{equation}
and

\begin{equation}\label{secondEqualityA}
\downarrow Y_{-}=\mathfrak{B}_k\dot{\bigcup}\left\{ v_{s+1},\ldots,v_{\beta_{k+1}}\right\}\dot{\bigcup} S_1^-(n,r)\dot{\bigcup} S_2^-(n,r)
\end{equation}

We prove at first (\ref{secondEqualityA}). By definition of $\mathfrak{B}_k$ and $T_-$ it easy to observe that $\mathfrak{B}_k\dot{\bigcup}\left\{ v_{s+1},\ldots,v_{\beta_{k+1}}\right\} \subseteq \downarrow T_-$ and moreover by lemma \ref{propertiesSublattices} ii) we also have that $\downarrow \theta=S_{1}^{-}(n,r)\bigcup S_{2}^{-}(n,r)$, hence, since $\downarrow Y_- = \downarrow \theta \bigcup \downarrow T_-$, it results that\\
$\mathfrak{B}_k\dot{\bigcup}\left\{ v_{s+1},\ldots,v_{\beta_{k+1}}\right\}\bigcup S_1^-(n,r)\dot{\bigcup} S_2^-(n,r) \subseteq \downarrow Y_-$.
On the other hand, by lemma \ref{propertiesSublattices} iv) we have that $\downarrow T_- \subseteq S_{1}^{\pm}(n,r)\dot{\bigcup} S_{1}^{-}(n,r)$ because $T_-$ is a subset of $S_{1}^{\pm}(n,r)$. At this point let us note that the elements of $\downarrow T_-$ that are also in $S_{1}^{\pm}(n,r)$ must belong necessarily to the subset $\mathfrak{B}_k\dot{\bigcup}\left\{ v_{s+1},\ldots,v_{\beta_{k+1}}\right\}$. This proves the other inclusion and hence (\ref{secondEqualityA}).\\
To prove now (\ref{firstEqualityA}) we must distinguish the cases (a1) and (a2). We set $\Delta := S_2^{\pm}(n,r)\dot{\bigcup} S_1^+(n,r) \dot{\bigcup} S_2^+(n,r)\bigcup \mathfrak{U}_k\dot{\bigcup}\left\{ v_1,\ldots,v_s\right\}$.
Let us first examine the case (a1).\\
Since $\alpha$ is the minimum of $S_{2}^{\pm}(n,r)$ we have $S_{2}^{\pm}(n,r) \subseteq \uparrow \alpha \subseteq \uparrow Y_+$. Moreover, since $Y_-=T_-\bigcup\{\theta\}$, it follows that
$\uparrow Y_-^c \supseteq \uparrow \theta^c = \uparrow (\Theta) = S_1^+(n,r)\dot{\bigcup} S_2^+(n,r)$ by lemma \ref{propertiesSublattices} i).
Finally, since $\mathfrak{U}_k\dot{\bigcup}\left\{ v_1,\ldots,v_s\right\} \subseteq \uparrow T_+ = \uparrow Y_+$ the inclusion $\supseteq$ in (\ref{firstEqualityA}) is proved. To prove the other inclusion $\subseteq$ we begin to observe that $\uparrow Y_+ \bigcap (S_1^-(n,r)\bigcup S_1^-(n,r))=\emptyset$, therefore the elements of $\uparrow Y_+$ that are not in $S_2^{\pm}(n,r)\dot{\bigcup} S_1^+(n,r) \dot{\bigcup} S_2^+(n,r)$ must be necessarily in $S_{1}^{\pm}(n,r)$, and such elements, by definition of $T_+$ must be necessarily in
$\mathfrak{U}_k\dot{\bigcup}\left\{ v_1,\ldots,v_s\right\}$. This proves that $\uparrow Y_+ \subseteq \Delta$.\\
For $\uparrow(Y_-^c)$, we have $\uparrow(Y_-^c)=\uparrow(T_-^c\bigcup\{\theta\}^{c})= (\uparrow T_-^c)\bigcup(\uparrow \Theta)$, where $\uparrow(\Theta)=S_1^+(n,r)\dot{\bigcup} S_2^+(n,r)$ by lemma \ref{propertiesSublattices} i) and, since $T_-\subseteq S_{1}^{\pm}(n,r)$, also
$T_-^c \subseteq (S_{1}^{\pm}(n,r))^c = S_{2}^{\pm}(n,r)$, by lemma \ref{propertiesSublattices} v). Therefore $\uparrow T_-^c \subseteq \uparrow S_{2}^{\pm}(n,r) \subseteq S_{2}^{\pm}(n,r) \bigcup S_{2}^{+}(n,r)$ by lemma \ref{propertiesSublattices} iii). This shows that
$\uparrow(Y_-^c) \subseteq S_{2}^{\pm}(n,r) \bigcup S_{2}^{+}(n,r) \bigcup S_{1}^{+}(n,r) \subseteq \Delta$, hence the inclusion $\subseteq$.
The proof of (\ref{firstEqualityA}) in the case (a1) is therefore complete.\\
Finally, to prove (\ref{firstEqualityA}) in the case (a2), it easy to observe that the only difference with to respect case (a1) is when we must show that $\uparrow Y_+ \subseteq \Delta$. In fact, in the case (a2) it results that $\alpha \notin \uparrow T_+$ and $Y_+=T_+\bigcup\{\alpha\}$, while $Y_-$ is the same in both cases (a1) and (a2). Therefore, in the case (a2), the elements of $\uparrow Y_+$ that are not in $S_2^{\pm}(n,r)\dot{\bigcup} S_1^+(n,r) \dot{\bigcup} S_2^+(n,r)$ must be in $(\uparrow T_+) \bigcap S_{1}^{\pm}(n,r)$ or in $(\uparrow \alpha)\bigcap S_{1}^{\pm}(n,r)$. As in the case (a1) we have $(\uparrow T_+) \bigcap S_{1}^{\pm}(n,r)= \mathfrak{U}_k\dot{\bigcup}\left\{ v_1,\ldots,v_s\right\}$ and, since $\alpha$ is the minimum of
$S_{2}^{\pm}(n,r)$, it results that $\uparrow \alpha = \uparrow S_{2}^{\pm}(n,r) \subseteq S_{2}^{\pm}(n,r) \bigcup S_{2}^{+}(n,r)$ by lemma \ref{propertiesSublattices} iii), hence $(\uparrow \alpha)\bigcap S_{1}^{\pm}(n,r)=\emptyset$. Therefore, also in the case (a2), the elements of $\uparrow Y_+$ that are not in $S_2^{\pm}(n,r)\dot{\bigcup} S_1^+(n,r) \dot{\bigcup} S_2^+(n,r)$ must be necessarily in $\mathfrak{U}_k\dot{\bigcup}\left\{ v_1,\ldots,v_s\right\}$. The other parts of the proof are the same as in the case (a1).
Hence we have proved the identities (\ref{firstEqualityA}) and (\ref{secondEqualityA}). By lemma \ref{basisKer(A)} it follows then that the map $A\in \mathcal W_{+}(S(n,r),{\bf 2})$. Finally, by definition of $A$, we have obviously $A(\theta)=N$, $A(\xi _1)=P$ and $A(\Theta)=P$. This shows that $A\in \mathcal{W}_{+}(n,r)$. The proof is complete.
\end{proof}

To conclude  we emphasize the elegant symmetry of the induced partitions on $S(n,r)$ from the boolean total maps $A_q$'s constructed in the proof of the theorem \ref{mainResult}.

\end{document}